\documentclass[12pt]{amsart}
\usepackage{amsfonts,amssymb,amscd,amstext,mathrsfs}

\input xy
\xyoption{all}

\newtheorem{theorem}{Theorem}[section]
\newtheorem{lemma}[theorem]{Lemma}

\newtheorem{proposition}[theorem]{Proposition}

\theoremstyle{definition}

\newcommand{\C}{\mathbb{C}}
\newcommand{\D}{\mathbb{D}}
\newcommand{\N}{\mathbb{N}}
\renewcommand{\P}{\mathbb{P}}
\newcommand{\R}{\mathbb{R}}

\newcommand{\cG}{\mathscr{G}}
\newcommand{\cO}{\mathscr{O}}
\newcommand{\cS}{\mathscr{S}}

\newcommand\Aut{\operatorname{Aut}}
\newcommand\Auta{\operatorname{Aut}_{\mathrm{alg}}}
\newcommand\Autaone{\operatorname{Aut}_{\mathrm{alg},1}}
\newcommand\Autone{\operatorname{Aut}_{1}}
\newcommand\Autsp{\operatorname{Aut}_{\mathrm{sp}}}
\newcommand\Autasp{\operatorname{Aut}_{\mathrm{alg,sp}}}

\newcommand\Id{\mathrm{id}}
\newcommand\wt{\widetilde}

\numberwithin{equation}{section}

\begin{document}
\title{Oka properties of groups \\ of holomorphic and algebraic automorphisms \\ of complex affine space}
\author{Franc Forstneri\v c and Finnur L\'arusson}

\address{Franc Forstneri\v c, Faculty of Mathematics and Physics, University of Ljubljana, and Institute of Mathematics, Physics, and Mechanics, Jadranska 19, 1000 Ljubljana, Slovenia}
\email{franc.forstneric@fmf.uni-lj.si}

\address{Finnur L\'arusson, School of Mathematical Sciences, University of Adelaide, Adelaide SA 5005, Australia}
\email{finnur.larusson@adelaide.edu.au}

\thanks{Forstneri\v c was supported by grants P1-0291 and J1-5432 from ARRS, Republic of Slovenia.  L\'arusson was supported by Australian Research Council grant DP120104110.}


\date{18 February 2014.  Minor changes 4 July 2014.}


\begin{abstract}
We show that the group of all holomorphic automorphisms of complex affine space $\C^n$, $n>1$, and several of its subgroups satisfy the parametric Oka property with approximation and with interpolation on discrete sets. 
\end{abstract}

\maketitle
\tableofcontents


\section{Introduction} 
\label{sec:intro}

\noindent
The holomorphic automorphisms of the complex line $\C$ are simply the affine maps $z\mapsto az+b$, $a,b\in \C$, $a\neq 0$.  For $n>1$, the group $\Aut\C^n$ of holomorphic automorphisms of $n$-dimensional complex affine space $\C^n$ is very large and complicated.  Holomorphic automorphisms of $\C^n$, $n>1$, have been intensively studied for the past 25 years, starting with Rosay and Rudin's seminal paper \cite{RR1988}.  A central result in this theory is the theorem of Anders\'en and Lempert, which states that every automorphism of $\C^n$, $n>1$, is a limit of compositions of special automorphisms called {\em shears} (for the definition of a shear, see (\ref{eq:shear}) below).  The infinitesimal version of the theorem is the Anders\'en-Lempert lemma, which states that every polynomial vector field on $\C^n$, $n>1$, is the sum of finitely many complete polynomial vector fields.  See \cite{A,AL}; also \cite{FR1993}, \cite[Chapter 4]{FF:book}, and \cite{KK2011}.

The study of holomorphic automorphisms of $\C^n$, $n>1$, has mainly focused on individual automorphisms and the various ways in which they can act on $\C^n$.  In this paper, we consider the whole group $\Aut\C^n$ as if it were a complex manifold.  It is not really an infinite-dimensional complex manifold, at least as far as we know, but there is a natural notion of a map $f$ from a complex manifold or a complex space $X$ into $\Aut\C^n$ being holomorphic.  It simply means that the associated map $X\times\C^n\to\C^n$, $(x,z)\mapsto f(x)(z)$, is holomorphic in the usual sense.  (The notion of a continuous map into $\Aut\C^n$ is straightforward, as $\Aut\C^n$ is naturally equipped with the compact-open topology.)

Complex Lie groups are the prototypical examples of Oka manifolds.  In particular, $\Aut\C$ is an Oka manifold.  Our thesis is that $\Aut\C^n$ behaves like an Oka manifold for $n>1$.  More precisely, we prove that $\Aut\C^n$ satisfies the parametric Oka property with approximation and with interpolation on discrete sets.  The parametric Oka property with approximation is one of the several mutually equivalent defining properties of an Oka manifold.  For the theory of Oka manifolds we refer the reader to the monograph \cite{FF:book} and the surveys \cite{F2013} and \cite{FL2011}.

The {\em parametric Oka property with approximation and interpolation} (POPAI) of a complex manifold $Y$ is defined as follows.  Let $X_0$ be a closed subvariety of a Stein space $X$.  (In this paper, all complex spaces are taken to be reduced.)  Let $K$ be a holomorphically convex compact subset of $X$.  Let $P_0\subset P$ be compact subsets of $\R^m$ (or sometimes more general parameter spaces).  For every continuous map $f:P\times X\to Y$ such that 
\begin{itemize}
\item $f|P_0\times X$ is $X$-holomorphic, meaning that $f(p,\cdot):X \to Y$ is holomorphic for every $p\in P_0$, and
\item $f(p,\cdot)$ is holomorphic on $K\cup X_0$ for every $p\in P$,
\end{itemize}
there is a continuous deformation $f_t:P\times X\to Y$, $t\in[0,1]$, of $f=f_0$ such that
\begin{itemize}
\item the deformation is fixed on $(P_0\times X)\cup (P\times X_0)$,
\item for every $t\in[0,1]$, $f_t(p,\cdot)$ is holomorphic on $K$ for every $p\in P$ and $f_t$ is uniformly close to $f$ on $P\times K$, and 
\item $f_1$ is $X$-holomorphic, that is, $f_1(p,\cdot):X\to Y$ is holomorphic for every $p\in P$.
\end{itemize}
Taking $X_0=\varnothing$ gives the {\em parametric Oka property with approximation} (POPA).  Taking $K=\varnothing$ gives the {\em parametric Oka property with interpolation} (POPI).  Taking $P$ to be a singleton and $P_0=\varnothing$ gives the {\em basic} version of each property (BOPAI, BOPA, BOPI).  BOPI implies that if a holomorphic map $X_0\to Y$ has a continuous extension $X\to Y$, then it has a holomorphic extension.  It is a major theorem that these six properties (and several others, including the consequence of BOPI just mentioned) are equivalent \cite{FF2009}, so they are referred to as the {\em Oka property}.  The Oka property is known to be homotopy-theoretic in the precise sense that it is equivalent to fibrancy in a certain model structure \cite{FL2004}.

Our results cover not only $\Aut \C^n$ itself, but the following six groups:
\begin{equation}\label{eq:groups}
\Aut\C^n, \quad \Autone\C^n, \quad \Autsp\C^n, \quad \Auta\C^n, \quad \Autaone\C^n, \quad \Autasp\C^n.
\end{equation}
Here, $\Autone\C^n$ consists of all holomorphic automorphisms of $\C^n$ with complex Jacobian $1$, $\Autsp\C^{n}$ for even $n=2m$ is the group of holomorphic automorphisms preserving the holomorphic symplectic form $\sum\limits_{j=1}^m dz_{2j-1}\wedge dz_{2j}$, $\Auta\C^n$ is the group of polynomial automorphisms of $\C^n$, 
\[ \Autaone\C^n=\Auta\C^n\cap \Autone\C^n, \]
and 
\[ \Autasp\C^n=\Auta\C^n\cap \Autsp\C^n. \]
We endow all these groups with the compact-open topology.  If $\cG$ is one of these groups, then POPAI with $Y$ above replaced by $\cG$ makes sense.  

Our main result is the following Oka principle.

\begin{theorem} \label{th:OP}
Let $n>1$.  The groups $\Aut\C^n$, $\Autone\C^n$, and $\Autsp\C^n$ satisfy the parametric Oka property with approximation and with interpolation on discrete sets.

The groups $\Auta\C^n$, $\Autaone\C^n$, and $\Autasp\C^n$ satisfy the parametric Oka property with approximation and with interpolation on finite sets.

The group $\Auta\C^2$ satisfies the parametric Oka property with approximation and with interpolation on discrete sets for maps whose values have bounded degree. 
\end{theorem}

Our proof uses Grauert's Oka principle for complex Lie groups, but not the modern Oka theory initiated by Gromov.  Anders\'en-Lempert theory is of course the key to proving approximation.  No such tool exists for interpolation.  We prove the following interpolation result with bare hands, except that for $\Auta\C^2$ we use some structure theory that partly fails and partly is unknown in higher dimensions.

\begin{theorem}\label{th:interpol}
Let $\cG$ be one of the subgroups of $\Aut\C^n$, $n>1$, on the list (\ref{eq:groups}) and let $X$ be a Stein space.  Let $D$ be a discrete subset of $X$ and let $f:D\to\cG$ be a map.  If $\cG$ is $\Aut\C^n$, $\Autone\C^n$, or $\Autsp\C^n$, then $f$ extends to a holomorphic map $X\to\cG$.  

If $\cG=\Auta\C^2$ and $\deg f:D\to\N$ is bounded, then $f$ extends to a holomorphic map $X\to\cG$.

If $\cG$ is $\Auta\C^n$, $\Autaone\C^n$, or $\Autasp\C^n$, and $D$ is finite, then $f$ extends to a holomorphic map $X\to\cG$.
\end{theorem}

Interpolation on arbitrary subvarieties $X_0$ of a Stein space $X$ has emerged as a challenging problem.

There is an interesting contrast between Theorem \ref{th:OP} and a known way in which the Oka principle for complex Lie groups can fail for $\Aut\C^n$, $n>1$.  By Demailly's work on the Serre problem \cite{D}, there is a holomorphic fibre bundle over $\C$ with fibre $\C^2$, whose total space is not Stein.  In particular, the bundle is not holomorphically trivial, even though it is topologically trivial since the base $\C$ is contractible.

Oka properties of $\Aut\C^n$ and its subgroups are related to the so-called {\em parametric Anders\'en-Lempert theorem}.  This point of view originated in the Forstneri\v c-Rosay theorem \cite{FR1993} on the approximation of smooth isotopies $F_t:\Omega\to \Omega_t$, $t\in [0,1]$, of biholomorphic maps between Runge domains in $\C^n$, with $F_0$ the identity map on $\Omega=\Omega_0$, by isotopies of holomorphic automorphisms of $\C^n$.  The proof is similar to that of the Anders\'en-Lempert theorem \cite{A,AL}, the main point being that any polynomial vector field on $\C^n$ is a sum of finitely many polynomial shear vector fields of the form (\ref{eq:SF}) and their $GL_n(\C)$-conjugates.  

A major generalisation appeared in the work of Varolin \cite{Varolin2001, Varolin2000}, who replaced $\C^n$ by a Stein manifold $S$ with the {\em density property}.  The density property means that the Lie algebra generated by all the complete holomorphic vector fields on $S$ is dense in the Lie algebra of all holomorphic vector fields on $S$.  See \cite[\S 4.10]{FF:book} and \cite{KK2011} for surveys of this subject.  A parametric version of the Anders\'en-Lempert theorem was suggested in \cite[Lemma 3.5]{Varolin2001}, which shows that if $V_x$ is a vector field on $S$ depending holomorphically on a parameter $x$ in a Stein space $X$, then $V_x$ can be approximated locally uniformly on $S$ by Lie combinations of complete vector fields which also depend holomorphically on $x$.  This was later made precise by Kutzschebauch \cite{K2005} when both $X$ and $S$ are affine spaces, and very recently by Kutzschebauch and Ramos-Peon \cite[Theorem 2]{KP2014} when $X$ is Stein and $S$ is Stein with the density property. 

It is easily seen that the parametric Anders\'en-Lempert theorem (say as presented in \cite{K2005}) implies that $\Aut\C^n$ satisfies a weak form of BOPA, in which $K$ is a convex compact subset of $X=\C^k$.  This property is called the {\em convex approximation property} and is known to imply the Oka property for complex manifolds \cite{FF2009}.  For $\Aut\C^n$, no such implication is known, and considerable additional work is required in order to prove our main result.


\section{Preliminaries}
\label{sec:prelim}

\noindent
In this section we recall some basic notions and prove a couple of preparatory results.  We follow the convention that a map is holomorphic on a compact set in a complex space if it is holomorphic on an unspecified open neighbourhood of the set.  When talking about a homotopy of such maps, it is understood that the neighbourhood is independent of the parameter.

Let $z=(z_1,\ldots,z_n)=(z',z_n)$ denote complex coordinates on $\C^n$.  A {\em shear} on $\C^n$ is an element of $\Aut\C^n$ which is $GL_n(\C)$-conjugate to an automorphism of one of the following two types:
\begin{equation}\label{eq:shear}
F(z)=\bigl(z', z_n+f(z')\bigr), \qquad H(z)=\bigl(z', e^{f(z')} z_n\bigr),
\end{equation}
where $f$ is an entire function on $\C^{n-1}$.  Shears of the first type are called {\em additive}, those of the second kind {\em multiplicative}. The composition of both types of shears is an 
{\em overshear}
\[ (z',z_n) \mapsto \left(z', e^{g(z')} z_n + f(z')\right).\] 
The shears (\ref{eq:shear}) are time-one maps in one-parameter subgroups of $\Aut\C^n$ given by
\begin{eqnarray}
F_t(z) &=& \bigl(z', z_n+tf(z')\bigr), \label{ssheargroup} \\
H_t(z) &=& \bigl(z', e^{t f(z')} z_n\bigr), \label{smsheargroup}
\end{eqnarray}
with $t\in\C$.  The infinitesimal generators of these groups are the entire vector fields
\begin{equation}\label{eq:SF}
V(z)=f(z')\, \frac{\partial}{\partial z_n},\qquad
W(z)=  f(z') z_n\, \frac{\partial}{\partial z_n}.
\end{equation}

Let $\cG$ be one of the subgroups of $\Aut\C^n$ on the list (\ref{eq:groups}).  We let $\cG_0=\cG\cap GL_n(\C)$ be the subgroup of $\cG$ consisting of the linear maps in $\cG$.  Thus $\cG_0$ is one of the groups $GL_n(\C)$,  $SL_n(\C)$, or $Sp_n(\C)$.  Note that for every $F\in \cG$, the derivative $D_z F$ of $F$ at any point $z\in \C^n$ also belongs to $\cG$, hence to $\cG_0$.   We can write every $F\in \cG$ as
\[ F(z)= a+ A H(z), \quad a=F(0)\in \C^n,\ A=D_0 F \in \cG_0, \]
where the automorphism $H\in \cG$ is of the form
\begin{equation}\label{eq:Schwarz}
H(z) = z + \sum_{\vert \alpha\vert=2}^\infty c_\alpha z^\alpha, \qquad  z\in \C^n, \ c_\alpha\in\C^n.
\end{equation}
Here, as usual, $z^\alpha=z_1^{\alpha_1}\cdots z_n^{\alpha_n}$.  We denote by $\cS\cG$ the subgroup of $\cG$ consisting of all elements of the form (\ref{eq:Schwarz}) and call it the {\em Schwarz subgroup of} $\cG$.  (The term is chosen by analogy with the Schwarz class on the disc $\D=\{z\in \C: |z|<1\}$, consisting of all injective holomorphic maps of the form $z\mapsto z+\sum\limits_{k=2}^\infty c_k  z^k$.)

Our next result says that for each of the groups $\cG$ on the list (\ref{eq:groups}), endowed with the compact-open topology, the linear subgroup $\cG_0$ carries the homotopy type of $\cG$, whereas the Schwarz subgroup $\cS\cG$ is contractible.

\begin{lemma}\label{lem:retract}
Let $\cG$ be one of the subgroups of $\Aut\C^n$ on the list (\ref{eq:groups}).  The linear subgroup $\cG_0$ is a strong deformation retract of $\cG$ in such a way that the Schwarz subgroup $\cS\cG$ retracts onto the trivial subgroup.  In particular, $\cS\cG$ is contractible.
\end{lemma}

\begin{proof}
Consider the map
\[ \cG \times (0,1] \to \cG, \quad (F,t)\mapsto \big(z\mapsto tF(0) + t^{-1}(F(tz)-F(0))\big), \]
with $(F,1)\mapsto F$.  It is continuous with respect to the compact-open topology on $\cG$ and extends continuously to a map $\cG \times [0,1] \to \cG$ with $(F,0)\mapsto D_0 F$.  Note that if $F$ is linear, then $(F,t)\mapsto F$ for all $t\in[0,1]$.  Also, if $F\in\cS\cG$, then the image of $(F,t)$ is in $\cS\cG$ for all $t\in [0,1]$, and $(F,0)\mapsto\Id$. 
\end{proof}
 
Note that if $f:Y\to\cG$ is a continuous map from a topological space $Y$ into one of the groups in (\ref{eq:groups}), then the induced map $D_0 f:Y\to\cG_0$, $y\mapsto D_0 f(y)$, is also continuous.

\begin{lemma}\label{lem:two-conditions}
Let $\cG$ be one of the subgroups of $\Aut\C^n$ on the list (\ref{eq:groups}).  Let $Y$ be a closed subspace of a topological space $X$, and $f:Y\to\cG$ be a continuous map.  If $f$ extends continuously to $X$, so does $D_0 f$.  Conversely, if $D_0 f$ extends continuously to $X$ and $Y\hookrightarrow X$ is a cofibration, then $f$ extends continuously to $X$. 
\end{lemma}

\begin{proof}
The first claim is evident.  As for the second, consider the continuous map 
\[ h:Y\to\cS\cG, \quad h(y)=(D_0 f(y))^{-1}\big(f(y)-f(y)(0)\big). \]
Since $Y\hookrightarrow X$ is a cofibration and $\cS\cG$ is contractible by Lemma \ref{lem:retract}, $h$ extends to a continuous map $H:X\to\cS\cG$.  Also, the map $y\mapsto f(y)(0)$ extends to a continuous map $O:X\to\C^n$.  And by assumption, $D_0 f$ extends to a continuous map $D:X\to\cG_0$.  Then $O+D H:X\to\cG$ is a continuous extension of $f$.
\end{proof}

If $X$ is a complex space and $Y$ is a closed subvariety of $X$, then $Y\hookrightarrow X$ is a cofibration.  If $U$ is a neighbourhood of a compact subset $K$ of a complex space $X$, then there is a neighbourhood $V$ of $K$ such that $\overline V\subset U$ and $\overline V\hookrightarrow X$ is a cofibration.  This follows from the triangulability of reduced complex spaces and their subvarieties \cite{Hardt1976, Hironaka1975}, and the fact that the inclusion of a subcomplex in a CW complex is a cofibration.


\section{The basic Oka property with approximation}
\label{sec:BOPA}

\noindent
In this section we prove that each of the groups $\cG$ in (\ref{eq:groups}) satisfies the basic Oka property with approximation.  Recall that given a holomorphic map $F:X\to \cG$, we denote by $D_0 F(x)\in \cG_0= \cG\cap GL_n(\C)$ the derivative of the holomorphic map $F(x):\C^n \to\C^n$ at $0\in \C^n$.

\begin{theorem}  \label{th:BOPA}
Let $\cG$ be one of the subgroups of $\Aut\C^n$, $n>1$, in (\ref{eq:groups}).  Assume that $X$ is a Stein space, $K$ is a compact $\cO(X)$-convex subset of $X$, and $F: U\to \cG$ is a holomorphic map on a neighbourhood $U$ of $K$ in $X$, such that the map $U\to\cG_0$, $x \mapsto D_0 F(x)$, extends to a continuous map $X\to \cG_0$.  Given a compact set $B\subset \C^n$ and a number $\epsilon>0$, there is a neighbourhood $U_0\subset U$ of $K$ and a homotopy of holomorphic maps $F_s: U_0\to \cG$, $s\in [0,1]$, such that $F_0=F\vert U_0$, $F_1$ extends to a holomorphic map $X\to \cG$, and 
\[  \sup\{\vert F_s(x,z)-F(x,z) \vert : x\in  K,\ z\in B,\ s\in [0,1]\} < \epsilon. \]
\end{theorem}

\begin{proof}
We begin by reducing to the case when $F$ is a map into the Schwarz subgroup $\cS\cG$ (\ref{eq:Schwarz}) of $\cG$.  The Taylor expansion of such a map is of the form 
\begin{equation}\label{eq:center0}
F(x,z) = z + \sum_{\vert \alpha\vert=2}^\infty c_\alpha(x) z^\alpha, \qquad  z\in \C^n,
\end{equation}
where the coefficients $c_\alpha(x)$ are holomorphic functions of $x\in U$ with values in $\C^n$.

Let $F: U\to \cG$ be as in the theorem.  For each $x\in U$, we set
\[ a(x) =F(x)(0)=F(x,0) \in \C^n,\qquad A(x)= D_0 F(x)\in \cG_0. \]
We then have 
\[  F(x)= a(x)+ A(x) H(x),\qquad x\in U,  \]
where $H:U \to \cS\cG$ is a holomorphic map of the form (\ref{eq:center0}).  By assumption, $A$ extends to a continuous map $A:X\to \cG_0$.  

By the Oka-Weil  theorem, we can approximate $a: U\to\C^n$ as closely as desired uniformly on $K$ by a holomorphic map $\tilde a:X\to\C^n$.  Also, since $\cG_0$ is a complex Lie group, Grauert's Oka principle \cite{Grauert1957,Grauert1958} implies that we can approximate $A: U\to \cG_0$ as closely as desired uniformly on $K$ by a holomorphic map $\wt A: X\to \cG_0$ which is homotopic to $A$ through a family of holomorphic maps $A_s : U\to \cG_0$, $s\in [0,1]$.

Assume for the moment that the theorem holds for maps into $\cS\cG$ of the form (\ref{eq:center0}). Let $\wt H: X\to \cS\cG$ be a holomorphic map which is homotopic to $H$ over a neighbourhood $U$ of $K$ in $X$ through a family of holomorphic maps $H_s: U\to \cS\cG$, $s\in [0,1]$, and such that every $H_s$ approximates $H_0=H$ uniformly on $K\times B$.  The homotopy of holomorphic map $F_s : U\to \cG$, $s\in [0,1]$, defined by 
\[ F_s(x) = (1-s)a(x) + s\tilde a(x) + A_s(x) H_s(x) \in \cG,\qquad x\in U,  \]
then satisfies the conclusion of the theorem. At $s=1$, we have the holomorphic map $F_1=\tilde a + \wt A \wt H:X\to \cG$.

This shows that it suffices to prove the result for maps into the Schwarz subgroup $\cS\cG$.  Assume that $F: U\to \cS\cG$ is a holomorphic map of the form (\ref{eq:center0}), where $U$ is a neighbourhood of $K$ in $X$. Consider the family of holomorphic maps
\[  F_t(x,z)= t^{-1} F(x,tz) =  z + \sum_{k=2}^\infty t^{k-1} \sum_{\vert \alpha\vert=k} c_\alpha(x) z^\alpha,\]
defined for $x\in U,\ z\in \C^n$, and $t\in \C$. Clearly $F_t(x,\cdot)\in \cS\cG$ for all $t\in \C$ and $x\in U$, and $F_0(x,\cdot)$ is the identity map on $\C^n$. Let 
\begin{equation}\label{eq:vectorfield}
V(t,x,z)= \sum_{|\alpha|=2}^\infty \sum_{i=1}^n v_{\alpha,i} (t,x) \, z^\alpha \frac{\partial}{\partial z_i}
\end{equation}
be the entire vector field on $\C^n$ defined by the flow equation
\begin{equation} \label{eq:floweq}
 \frac{\partial}{\partial t} F_t(x,z) = V(t,x,F_t(x,z)).
\end{equation}
This means that the one-parameter group $\C\to\Aut\C^n$, $t\mapsto F_t(x,\cdot)$, is the flow of $V$; in particular, $F(x,z)=F_1(x,z)$ is obtained by integrating the vector field $V$ over the time interval $t\in [0,1]$, beginning at $t=0$ with the identity map on $\C^n$.  The coefficients $v_{\alpha,i}(t,x)$ of $V$ depend holomorphically on $t\in \C$ and $x\in U$. 

At this point we follow the proof of the Anders\'en-Lempert theorem \cite{AL} as given in \cite{FR1993} and in \cite[\S 4.9]{FF:book}.  (See also \cite{K2005} and  \cite[Theorem 2]{KP2014}.)

We first subdivide the time interval $[0,1]$ into a large number $N$ of subintervals $I_j=[j/N,(j+1)/N]$ for $j=0,1,\ldots,N-1$.  For $t\in I_j$, we replace $V(t,x,z)$ by the autonomous vector field $V(j/N,x,z)$.  If $N$ is large enough, this replacement amounts to an arbitrarily small error in the time-one map of the flow, with uniform estimates on $K\times B$ for any given compact set $B\subset \C^n$.  More precisely, the error at each step is $o(1/N)$, so the total error is $N o(1/N)$, which goes to $0$ as $N\to\infty$.  See \cite[Theorem 4.8.2]{FF:book}.  Hence, to prove Theorem \ref{th:BOPA}, it suffices to consider flows defined by  entire vector fields on $\C^n$ of the form (\ref{eq:vectorfield}) that are autonomous, that is, with the coefficients $v_{\alpha,i} (t,x)$ independent of time $t$.

Next we cut off the series  (\ref{eq:vectorfield}) at some finite-order term.  This amounts to an arbitrarily small error term in the time-one map of the flow provided the degree is sufficiently large. We thus obtain a polynomial vector field
\begin{equation}\label{eq:PVF}
W(x,z)= \sum_{2\leq|\alpha|\leq k}  \sum_{i=1}^n w_{\alpha,i} (x) \, z^\alpha \frac{\partial}{\partial z_i},
\end{equation}
whose coefficients $w_{\alpha,i} (x)$ are holomorphic functions on a neighbourhood of $K$ in $X$.

The remainder of the proof somewhat depends on the group $\cG$.  Assume first that $\cG= \Aut\C^n$.  By the Anders\'en-Lempert  lemma (see \cite{AL} or \cite[Proposition 4.9.3]{FF:book}), each vector field $z^\alpha \partial/\partial z_i$ in (\ref{eq:PVF}) is the sum of finitely many polynomial shear vector fields (\ref{eq:SF}) on $\C^n$.  (The flows of these vector fields are complex one-parameter groups in $\Aut\C^n$ that are $GL_n(\C)$-conjugate to one-parameter groups of the form (\ref{ssheargroup}) and (\ref{smsheargroup}).)  This gives a representation of $W$ as a finite sum 
\begin{equation}\label{eq:PVF2}
W(x,z)= \sum_j  b_j(x) W_j(z),
\end{equation}
where each $W_j$ is a shear vector field on $\C^n$ (conjugate to a vector field of the form (\ref{eq:SF})) and $b_j$ is a holomorphic function on a neighbourhood of $K$ in $X$.

Observe that if $W$ is a complete holomorphic vector field with flow $\phi_t(z)$, $t\in \C$, then for any complex-valued function $b(x)$ of a parameter $x$, the product $b(x) W$ is still a complete vector field with the flow $\phi_{t b(x)}(z)$ for each $x$. If $b$ is holomorphic in $x$, then so is the flow $\phi_{t b(x)}(z)$. This observation applies to every summand $b_j(x) W_j$ in (\ref{eq:PVF2}). 

We now apply  the Oka-Weil theorem to approximate each of the coefficients $b_j$ in (\ref{eq:PVF2}), uniformly on $K$, by a holomorphic function $\wt b_j$ on $X$.  This gives a polynomial vector field $\wt W(x,z)= \sum_j \wt b_j(x) W_j(z)$ with coefficients $\wt b_j\in\cO(X)$, which approximates $W$ uniformly on $K$.  A suitable concatenation of flows of the shear vector fields  $\wt b_j(x) W_j(z)$ furnishes automorphisms $\wt F(x)\in \cS\Aut\C^n$ of the form
\[ \wt F(x) = \Phi_1(x) \circ \Phi_2(x) \circ \cdots \circ\Phi_N(x),\qquad x\in X, \]
such that each $\Phi_j:X\to \Aut\C^n$ is a holomorphic map into the polynomial shear subgroup of $\Aut\C^n$, and $\wt F$ approximates $F$ on $K$.  (The fact that $\tilde F$ maps into the Schwarz subgroup follows from the vanishing of our vector fields to second order at $z=0$; however, all we need here is an approximation of $F$ with values in the whole group.)  A homotopy from $\wt F$ to $F$ is obtained by introducing a parameter in each of the shear groups which at time 1 give the maps $\Phi_1,\ldots,\Phi_N$.  The details can be found in \cite[\S 4.9]{FF:book}.  This completes the proof of Theorem \ref{th:BOPA} when $\cG=\Aut\C^n$.

Suppose next that $\cG=\Autone\C^n$. In this case the vector field $V(t,x,\cdot)$ on $\C^n$ (\ref{eq:vectorfield}) for fixed $t\in \C$ and $x\in U$  has divergence zero with respect to the standard holomorphic volume form $\omega=dz_1\wedge\cdots\wedge dz_n$ on $\C^n$. Since every vector field $W$ (\ref{eq:PVF}) under consideration, before being replaced by one of its Taylor polynomials, is of the form $V(t_j,\cdot)$ for some $t_j=j/N$, it also has divergence zero.  Recall that the divergence equals 
\[ \operatorname{div}_\omega \left( \sum_{i=1}^n c_i(z) \frac{\partial}{\partial z_i} \right) = \sum_{i=1}^n \frac{\partial c_i(z)}{\partial z_i}. \]
Applying this to the vector field $W$ (\ref{eq:PVF}) we get
\[ \operatorname{div} W(x,z) = \sum_{k\ge 2} \sum_{\vert \alpha\vert =k} \sum_{i=1}^n w_{\alpha,i} (x) \, \alpha_i z^{\alpha-e_i} =0. \]
Here $e_i$ denotes the $n$-vector whose $i$-th entry equals $1$ and all other entries equal $0$. It follows that the sum of the homogeneous terms of each degree $k$ in the Taylor expansion of $W$, and hence any Taylor polynomial of $W$, has divergence zero for each fixed $x\in U$.  By Anders\'en's lemma (see \cite{A} or \cite[Lemma 4.9.5]{FF:book}), the finite-dimensional complex vector space of all homogeneous divergence-zero vector fields on $\C^n$ of any fixed degree $k$ admits a basis consisting of shear vector fields of the form $f(z')\partial/\partial z_n$ and their $SL_n(\C)$-conjugates.  If a vector field in this space depends holomorphically on a parameter $x$, then its coefficients with respect to this basis will also depend holomorphically on $x$, since they are $\C$-linear combinations of the coefficients of the same vector field when expressed in the standard basis for $\C^n$. Hence the vector field $W$ (\ref{eq:PVF}) is a finite linear combination $W(x,z) = \sum_{j} g_j(x) W_j(z)$ of divergence-zero polynomial vector fields $W_j$ on $\C^n$ with coefficients $g_j(x)$ that are holomorphic functions of $x \in U\subset X$.  We complete the proof exactly as before, approximating the coefficients $g_j$ by holomorphic functions $\tilde g_j$ on $X$, and then approximating the flow of the vector field $\wt W(x,z) = \sum_{j} \tilde g_j(x) W_j(z)$ by compositions of shears of the form (\ref{ssheargroup}) and their $SL_n(\C)$-conjugates.  This completes the proof of Theorem \ref{th:BOPA} for $\cG=\Autone\C^n$.

Every element $F\in \Auta\C^n$ of the polynomial automorphism group has constant Jacobian.  Using the Oka principle for maps into $\C^*$, it is easy to reduce to the case when the Jacobian equals $1$.  We then proceed as in the case $\cG=\Autone\C^n$, approximating $F$ by a holomorphic map $\wt F: X\to \Autaone\C^n$ with image in the polynomial shear subgroup of $\Autaone\C^n$.

If $n=2m$ is even and $\cG$ is $\Autsp\C^n$ or $\Autasp\C^n$, the vector fields $V$ (\ref{eq:vectorfield}) and $W$ (\ref{eq:PVF}) are Hamiltonian with respect to the symplectic form $\omega=\sum\limits_{j=1}^m  dz_{2j-1}\wedge dz_{2j}$ (see \cite{FF1996}).  By essentially the same argument as for $\Autone\C^n$, using the fact that a polynomial Hamiltonian vector field can be written as a finite sum of polynomial Hamiltonian shear vector fields \cite[Proposition 5.2]{FF1996}, we complete the proof just as in the volume-preserving case.
\end{proof}


\section{The parametric Oka property with approximation}
\label{sec:POPA}

\noindent
In this section we prove the following parametric version of Theorem \ref{th:BOPA}.

\begin{theorem} \label{th:POPA}
Let $\cG$ be one of the subgroups of $\Aut\C^n$, $n>1$, in (\ref{eq:groups}), and let $\cG_0= \cG\cap GL_n(\C)$ be its linear subgroup.   Assume that $X$ is a Stein space, $K$ is a compact $\cO(X)$-convex subset of $X$, $U$ is a neighbourhood of $K$ in $X$, $P_0\subset P$ are compact Hausdorff spaces such that $P_0$ is a strong deformation retract of a neighbourhood in $P$, and $F: (P \times U)\cup (P_0\times X) \to \cG$ is a continuous map such that
\begin{enumerate}
 \item[\rm (i)]     $F(p,\cdot) : U \to \cG$ is holomorphic for every $p\in P$, 
 \item[\rm (ii)]    $F(p,\cdot):X\to \cG$ is holomorphic for every $p\in P_0$, and 
 \item[\rm (iii)]   $(P \times K)\cup (P_0\times X)\to\cG_0$, $(p,x) \mapsto D_0 F(p,x)$, extends to a continuous map $A:P\times X\to \cG_0$.
\end{enumerate}
Given a compact set $B\subset \C^n$ and a number $\epsilon>0$, there is a neighbourhood $U_0\subset U$ of $K$ and a homotopy $F_s: (P \times U_0)\cup (P_0\times X) \to \cG$, $s\in [0,1]$, with the same properties as $F=F_0$, such that
\begin{enumerate}
 \item   $F_1(p,\cdot): X\to \cG$ is holomorphic for every $p\in P$, 
 \item   $F_s(p,x)=F(p,x)$ for every $(p,x) \in P_0\times X$ and $s\in [0,1]$, and 
 \item   $\sup\{\vert F_s(p,x,z)-F(p,x,z) \vert : s\in [0,1],\ p\in P,\ x\in  K,\ z\in B\} < \epsilon.$
\end{enumerate}
\end{theorem}

\begin{proof}
As in the proof of Theorem \ref{th:BOPA}, we reduce to the case of maps
\[  F(p,x,z) = z + \sum_{\vert \alpha\vert=2}^\infty c_\alpha(p,x) z^\alpha, \qquad  z\in \C^n, \]
with values in the Schwarz subgroup $\cS\cG$, where $(p,x)\in  (P \times U)\cup (P_0\times X)$. (The neighbourhood $U$ of $K$ may shrink in the course of the proof.)  The coefficients $c_\alpha(p,x)$ are holomorphic in $x$ where defined.  The reduction is accomplished by applying a parametric version of Grauert's Oka principle to the center map $a(p,x)=F(p,x,0)\in \C^n$ and the derivative map $D_0 F(p,x)\in \cG_0$ with respect to the pair of parameter spaces $P_0\subset P$ (see e.g.\ \cite[Theorem 5.4.4]{FF:book}), noting once again that a complex Lie group is an Oka manifold.

Next we observe that the proof of Theorem \ref{th:BOPA} applies verbatim to the present situation if we replace $X$ by $P\times X$ and consider maps  to the group $\cS\cG$ that are $X$-holomorphic, provided that the interpolation condition (2) in Theorem \ref{th:POPA} is replaced by the following condition:
\begin{itemize}
\item[\rm (2')]  $F_s(p,x,z)=z$ for every $p\in P_0$, $x\in X$, $z\in \C^n$, and $s\in [0,1]$.
\end{itemize}
The conclusion provided by the proof of Theorem \ref{th:BOPA} in this situation is that any $X$-holomorphic map $F: P\times U\to \cS\cG$, where $U\subset X$ is a neighbourhood of $K$, satisfying condition (2'), can be approximated uniformly on $P\times K\times B$ by an $X$-holomorphic map $\wt F : P\times X\to \cS\cG$, which is homotopic to $F$ on $(P_0\times X)\cup (P\times U)$ through a family of maps $F_s$ satisfying condition (2').

To see this, note that the entire vector field $V(t,p,x,\cdot)$ on $\C^n$ in the proof of Theorem \ref{th:BOPA} (see (\ref{eq:vectorfield}) and (\ref{eq:floweq})) vanishes for all $p\in P_0$ since the homotopy $F_t$ is fixed there.  The same is then true for the shear vector fields $W_j(p,\cdot)$ (\ref{eq:PVF2}) obtained in the proof. Thus the new approximating homotopy, obtained by composing the flows of these shear vector fields, is also fixed for all $p\in P_0$, and hence equal to the identity map on $\C^n$.

We now turn to the general case. By assumption, there are a neighbourhood $V_0\subset P$ of $P_0$ and a deformation retraction $\tau_s=\tau(s,\cdot) : V_0\to V_0$, $s\in [0,1]$, such that $\tau_0$ is the identity on $V_0$, $\tau_1(V_0)=P_0$, and $\tau(s,p)=p$ for $p\in P_0$ and $s\in [0,1]$.  Choose small neighbourhoods $V\Subset V_1 \Subset V_0$ of $P_0$ and a continuous function $\rho: P\to [0,1]$ such that $\rho=0$ on $V$ and $\rho=1$ on $P\setminus V_1$.  Define a map 
\[ H: (V\times X)\cup (P\times U)\to \cS\cG, \qquad H(p,x,z) = F\bigl( \tau(1-\rho(p),p), x, z). \]
When $p\in V$, we have $\tau(1-\rho(p),p)=\tau(1,p) \in P_0$, so $H$ is well defined. Clearly $H$ is 
continuous and $X$-holomorphic, $H(p,\cdot,\cdot)=F(p,\cdot,\cdot)$ for all $p\in P_0$, and for any compact set $B'\subset \C^n$, the map $H$ is close to $F$ uniformly on $P\times K\times B'$ when the neighbourhood $V_1$ of $P_0$ is chosen small enough.  (The set $B'$ used here may have to be much bigger than the original set $B\subset\C^n$ in the theorem, since the final homotopy $F_s$ will be a composition of several homotopies, and we must ensure sufficiently good approximation on the range of every term in the composition except for the last one.)

Consider the homotopy defined by 
\[ H_t(p,x,z)=t^{-1} H(p,x,tz),\quad (p,x)\in (V\times X)\cup (P\times U), \quad t\in\C. \]
Note that  $H_0(p,x,z)=z$.  Pick a continuous function $\xi : P\to [0,1]$ such that $\xi=0$ on $P_0$ and $\xi=1$ on $P\setminus V$.  Clearly,
\begin{equation}\label{eq:splitting}
H(p,x,\cdot) = H(p,x,\cdot)\circ H_{\xi(p)}(p,x,\cdot)^{-1} \circ H_{\xi(p)}(p,x,\cdot)  \in \cS\cG
\end{equation}
for all $(p,x)\in  (V\times X)\cup (P\times U)$. Observe that the map
\[ \Psi(p,x,\cdot)= H(p,x,\cdot)\circ H_{\xi(p)}(p,x,\cdot)^{-1} \in \cS\cG \]
is defined and $X$-holomorphic on $P\times X$.  This is clear for $p\in V$, while for $p\in P\setminus V$, we have $\xi(p)=1$ and hence $\Psi(p,x,z)=z$. 

The last term $H_{\xi(p)}(p,x,\cdot)$ on the right-hand side of (\ref{eq:splitting}) is defined and $X$-holo\-morphic on the set $P\times U$, and $H_{\xi(p)}(p,x,z)=z$ for all $(p,x)\in P_0 \times U$, since $\xi(p)=0$ for $p\in P_0$.  By  the  observation made at the beginning of the proof, this term can be approximated by an $X$-holomorphic map $\Theta : P\times X\to \cS\cG$ whose value is the identity map on $\C^n$ for all points $(p,x)\in P_0\times X$. Setting 
\begin{equation}\label{eq:solution}
\wt H(p,x,\cdot) = H(p,x,\cdot)\circ H_{\xi(p)}(p,x,\cdot)^{-1} \circ \Theta(p,x,\cdot)  \in \cS\cG,
\end{equation}
we get an $X$-holomorphic map $\wt H : P\times X\to \cS\cG$ which agrees with $H$ on $P_0\times X$ and approximates $H$ on $P\times K$. Furthermore, there is a homotopy $\Theta_s : (V\times X)\cup (P\times U)\to \cS\cG$, $s\in [0,1]$, of holomorphic maps such that $\Theta_0(p,x,\cdot)=H_{\xi(p)}(p,x,\cdot)$, $\Theta_1=\Theta$, $\Theta_s(p,x)=\Id$ for $(p,x)\in P_0 \times U$, and $\Theta_s$ approximates $\Theta_0$ uniformly on $P\times K$.  Inserting $\Theta_s$ in place of $\Theta$ in (\ref{eq:solution}) gives a homotopy $H_s : (V\times X)\cup (P\times U)\to \cS\cG$, $s\in [0,1]$, from $H_0=H$ to $H_1=\wt H$ with the desired properties. This completes the proof.
\end{proof}


\section{Interpolation by entire curves of automorphisms}  
\label{sec:interpolation}

\noindent
By the fact that the values of holomorphic functions on Stein spaces can be arbitrarily prescribed on discrete sets, Theorem \ref{th:interpol} is an immediate corollary of the following lemma on interpolation by entire curves, except for the case of $\Auta\C^2$, which is treated in Section \ref{sec:plane}.

\begin{lemma} \label{lem:entire}
Let $\cG$ be one of the groups  $\Aut\C^n$,  $\Autone\C^n$, or $\Autsp\C^n$.  Given a sequence $(\phi_k)_{k\in \N}$ in $\cG$, there is a holomorphic map $F : \C\to \cG$ satisfying $F(k)=\phi_k$ for every $k\in\N$.  If $\cG$ is one of the groups $\Auta\C^n$, $\Autaone\C^n$, or $\Autasp\C^n$, then the interpolation is possible at finitely many points.
\end{lemma}

\begin{proof}
Recall that $\cG_0=\cG\cap GL_n(\C)$ and $\cS\cG$ is the subgroup  consisting of all $F\in \cG$ of the form
\[ F(z) = z + \sum_{\vert \alpha\vert=2}^\infty c_\alpha z^\alpha, \qquad  z\in \C^n. \]
Let $\phi_k(0)=a_k\in\C^n$ and $D_0\phi_k(0)=A_k\in \cG_0$ for $k\in \N$.  Then we have 
\[ \phi_k(z) = a_k + A_k\psi_k(z),\qquad z\in \C^n,\ \psi_k\in \cS\cG,\ k\in \N. \]
By the Oka property with interpolation for maps to complex Lie groups, there are entire maps $a: \C\to \C^n$ and $A: \C\to \cG_0$ such that $a(k)=a_k$ and $A(k)=A_k$ for all $k\in \N$.  Hence it suffices to find an entire map $H: \C\to \cS\cG$ satisfying $H(k)= \psi_k$ for all $k\in \N$; then the map $F:\C\to \cG$, defined by 
\[ F(x)(z) = a(x) + A(x) H(x)(z),\qquad x\in \C,\ z\in \C^n, \]
satifies the conclusion of the lemma.

We shall construct $H$ as a limit of entire maps $H_k : \C\to \cS\cG$ satisfying $H_k(j)=\psi_j$ for $j=1,\ldots,k$, such that the sequence $(H_k)$ converges uniformly on compacta in $\C\times \C^n$.  If $\cG$ is one of the groups $\Auta\C^n$, $\Autaone\C^n$, or $\Autasp\C^n$, then we are unable to pass to the limit within $\cG$, and we only obtain interpolation at finitely many points.

For $F\in \cS\cG$ as above we define 
\[  F^t(z) =  t^{-1} F(tz) = z + \sum_{\vert \alpha\vert=2}^\infty t^{\vert \alpha\vert -1} c_\alpha z^\alpha, \qquad  z\in \C^n,\ t\in \C. \]
The map $t\mapsto F^t \in \cS\cG$ is an entire curve of automorphisms with $F^0=\Id$ and $F^1=F$.

We start by setting $H_1(t)=\psi_1^t$ for $t\in \C$.  Clearly, $H_1(0)=\Id$ and $H_1(1)=\psi_1$.  Fix a number $\epsilon_1>0$ and choose an entire function $h_1\in\cO(\C)$ such that 
\[ h_1(0)=h_1(1)=0, \quad h_1(2)=1,\quad  |h_1(x)|<\epsilon_1 \text{ when}\ |x|\le 1. \]
Set $\theta_{2}= \psi_2\circ H_1(2)^{-1} \in \cS\cG$ and let
\[ H_2(x) =  \theta^{h_1(x)}_{2}  \circ H_1(x) \in \cS\cG,\qquad x\in \C. \]
We have $H_2(0)=H_1(0)=\Id$ and $H_2(1)=H_1(1)=\psi_1$ (since $h_1(0)=h_1(1)=0$), and $H_2(2)=  \theta_{2} \circ H_1(2) = \psi_2$. Furthermore, as $h_1$ is small on the disc where $|x|\le 1$, $H_2$ is close to $H_1$ there.  Next we pick a number $\epsilon_2>0$ and an entire function $h_2\in\cO(\C)$ such that 
\[ h_2(0)=h_2(1)=h_2(2)=0, \quad h_2(3)=1,\quad  |h_2(x)|<\epsilon_2 \text{ when}\ |x|\le 2. \]
Set $\theta_{3}= \psi_3\circ H_2(3)^{-1} \in \cS\cG$ and
\[ H_3(x) =  \theta^{h_2(x)}_{3}  \circ H_2(x)\in \cS\cG,\qquad x\in \C. \]
As before, we verify that the entire map $H_3: \C\to \cS\cG$ assumes the correct values at the points $0,1,2,3$, and is close to $H_2$ on the disc where $|x|\le 2$ since $h_2$ is small there.  It is clear how the construction continues.  If $\epsilon_k\to 0$ sufficiently fast, then $(H_k)$ converges uniformly on compacta in $\C\times \C^n$ to an entire map $H: \C\to \cS\cG$ with $H(k)= \psi_k$ for all $k\in \N$.
\end{proof}


\section{The Oka property with approximation and interpolation} 
\label{sec:BOPAI}

\noindent
In this section we show that our approximation results, Theorems \ref{th:BOPA} and \ref{th:POPA}, can be combined with interpolation on finite or infinite discrete sets (Theorem \ref{th:interpol}).  This proves Theorem \ref{th:OP}.  We begin with the non-parametric case.

\begin{theorem} \label{th:BOPAI}
Let $\cG$ be one of the groups $\Aut\C^n$,  $\Autone\C^n$, or $\Autsp\C^n$.  Assume that $X$ is  a Stein space, $K$ is a compact $\cO(X)$-convex subset of $X$, $(x_k)_{k\in\N}$ is a discrete sequence without repetition in $X$, $(\phi_k)_{k\in \N}$ is a sequence in $\cG$, and $F: K\to \cG$ is a holomorphic map such that
\begin{itemize}
\item[\rm (i)]  $F(x_k)=\phi_k$ for all $x_k\in K$, and
\item[\rm (ii)]  the map $x\mapsto D_0 F(x)$ from a neighbourhood of $K$ into $\cG_0$ extends to a continuous map $X\to \cG_0$.
\end{itemize} 
Given a compact set $B\subset \C^n$ and number $\epsilon>0$, there is a holomorphic map $\wt F: X\to \cG$, homotopic to $F$ over $K$ through a homotopy that is fixed at the points $x_k\in K$, such that 
\begin{itemize}
\item[\rm (a)]  $\wt F(x_k)=\phi_k$ for all $k\in \N$, and 
\item[\rm (b)]  $\sup\bigl\{\vert \wt F(x,z)-F(x,z) \vert : x\in  K,\ z\in B\bigr\} < \epsilon$.
\end{itemize} 
The result holds for $\cG=\Auta\C^2$ if the sequence $(\deg\phi_k)$ is bounded.  If $\cG$ is one of the groups $\Auta\C^n$, $\Autaone\C^n$, or $\Autasp\C^n$, then the result holds with interpolation at finitely many points of $X$.
\end{theorem}

\begin{proof}
We treat the case when $\cG$ is $\Aut\C^n$,  $\Autone\C^n$, $\Autsp\C^n$, or $\Auta\C^2$.  By Theorem \ref{th:interpol}, there is a holomorphic map $G: X \to \cG$ satisfying $G(x_k)=\phi_k$ for all $k\in \N$. The holomorphic map $H: K\to \cG$, $x\mapsto F(x)\circ G(x)^{-1}$, then satisfies $H(x_k)=\Id$ for all $k$ such that $x_k\in K$.  By Lemma \ref{lem:two-conditions}, the map $x\mapsto D_0 H(x)$ extends from a neighbourhood of $K$ to a continuous map $X\to \cG_0$ (we choose a neighbourhood $U$ of $K$ such that $\overline U$ is contained in the neighbourhood in (ii) and $\overline U\hookrightarrow X$ is a cofibration).  Since $K$ is compact, the set $\bigcup\limits_{x\in K} G(x)(\bar B)\subset \C^n$ is compact and hence is  contained in some ball  $B'\subset \C^n$. We look for a holomorphic map $\wt H: X\to \cG$ satisfying the following two conditions.
\begin{itemize}
\item[\rm (a')]  $\wt H(x_k)=\Id$ for all $k\in \N$.
\item[\rm (b')]  $\sup\bigl\{\vert \wt H(x,z)-H(x,z) \vert : x\in  K,\ z\in B'\bigr\} < \epsilon$.
\end{itemize} 
These conditions clearly imply that the map $\wt F = \wt H\circ G: X\to \cG$ satisfies properties (a) and (b) in Theorem \ref{th:BOPAI}.

We find a map $\wt H$ with these properties by following the proof of Theorem \ref{th:BOPA}, paying attention to the additional requirement that we interpolate the identity map at the points $x_k$.  We will go through the individual steps to see that this is possible.  The reader should  keep in mind that $H$ now plays the same role as $F$ does in the proof of Theorem \ref{th:BOPA}.

The first step is to adjust the affine part of the map $H$, thereby reducing the problem to the case when $H$ takes values in the Schwarz subgroup $\cS\cG$ (\ref{eq:Schwarz}). This is done as in the proof of Theorem \ref{th:BOPA} using the Oka principle for maps from Stein spaces to complex Lie groups. 

Hence we may assume that $H$ takes values in $\cS\cG$.  We replace the ball $B'$ above by a suitably bigger ball $B''\subset \C^n$ in what follows. Consider the homotopy $H_t(x,z)=t^{-1}H(x,tz)$ for $t\in \C$, with $H_0(x,z)=z$, and define the corresponding infinitesimal generator $V(t,x,z)$ as in (\ref{eq:vectorfield}).  At points $x_k\in K$, we have $H_t(x_k,z)=z$ for all $t\in \C$, and hence $V(t,x_k,z)=0$ for all $t\in\C$ and $z\in \C^n$.  Thus the coefficients $v_{\alpha,i}(t,x)$ of $V$ vanish at the points $x=x_k\in K$. As explained in the proof of Theorem \ref{th:BOPA}, we can preserve this condition at all subsequent steps of the proof, which are as follows.
\begin{itemize}
\item[\rm (a)]  Replace the  infinitesimal generator $V$ by autonomous vector fields on short time intervals in $[0,1]$.
\item[\rm (b)]  Approximate these autonomous vector fields by polynomial vector fields on $\C^n$ depending holomorphically on $x\in X$, see (\ref{eq:PVF}).
\item[\rm (c)]  Decompose each polynomial vector field as a finite sum of polynomial shear vector fields depending holomorphically on $x\in X$,  see (\ref{eq:PVF2}).
\item[\rm (d)]  Concatenate the flows of the polynomial shear vector fields obtained in step (c) to obtain the desired approximation.
\end{itemize}
In step (b), when approximating the coefficients of the polynomial vector fields uniformly on $K$ by holomorphic functions on $X$, we can ensure that these functions vanish on the discrete sequence $(x_k)$ by a standard result of Stein space theory.  In step (c), we can also obtain shear vector fields that depend holomorphically on $x \in X$ and vanish on $(x_k)$ (see the proof of Theorem \ref{th:BOPA}).  The automorphisms of $\C^n$ obtained from their flows then approximate $H$ on $K$ and match the identity on $\C^n$ at each point $x_k$. 
\end{proof}

Our main result, Theorem \ref{th:OP}, is the parametric version of Theorem \ref{th:BOPAI}.  The proof is essentially the same as that of Theorem \ref{th:BOPAI}.  We first use Theorem \ref{th:interpol} to reduce to the case when $\phi_k=\Id$ for all $k$, and then we apply the parametric versions of the  arguments in the proof of Theorem \ref{th:BOPAI}.  We omit the details.

We conclude this section by formulating an extension property that is a necessary and sufficient condition for any of our six groups to satisfy the full parametric Oka property with approximation and interpolation.

\begin{theorem} \label{th:extension}
Let $\cG$ be one of the subgroups of $\Aut\C^n$, $n>1$, in (\ref{eq:groups}).  The following are equivalent.
\begin{enumerate}
\item[(i)]  Let $X_0$ be a closed subvariety of a Stein space $X$.  If a holomorphic map $X_0 \to \cG$ has a continuous extension $X \to \cG$, then it has a holomorphic extension.
\item[(ii)]  Let $X_0$ be a closed subvariety of a Stein space $X$.  Every holomorphic map $X_0\to\cS\cG$ extends to a holomorphic map $X\to\cS\cG$.
\item[(iii)]  $\cG$ satisfies POPAI.
\end{enumerate}
\end{theorem}


\section{Algebraic automorphisms of the affine plane} 
\label{sec:plane}

\noindent
In this section we complete the proof of Theorem \ref{th:interpol} by explaining how to find holomorphic maps from a Stein space $X$ into the group $\cG=\Auta\C^2$ with prescribed values of bounded degree on an infinite discrete subset.  Note that if $X$ has finitely many irreducible components, then bounded degree is a necessary condition for extendibility.  Namely, if $f:X\to\cG$ is holomorphic, then $X$ is the increasing union of the closed subvarieties $Y_n$, $n\in\N$, where $\deg f\leq n$, so $Y_n=X$ for $n$ sufficiently large.

We recall the basic facts of the structure theory of $\cG$ (see \cite[Section 2]{FM1989}).  By the Jung-van der Kulk theorem, $\cG$ is the amalgamated free product of the subgroups $A$ of affine automorphisms and $E$ of elementary automorphisms, that is, automorphisms of the form $(x,y)\mapsto (ax+p(y),by+c)$, where $a,b,c\in\C$, $ab\neq 0$, and $p$ is a polynomial.  Every $g\in\cG$ that does not belong to the intersection $S=A\cap E$ may be expressed as a composition $g_n\circ\cdots\circ g_1$, where each $g_i$ belongs to either $A$ or $E$ but not to $S$, and no two consecutive factors belong to the same subgroup $A$ or $E$.  This expression for $g$ is unique, except that for every $s\in S$ and every $i>1$, we may replace $g_i$ by $g_i\circ s$ and $g_{i-1}$ by $s^{-1}\circ g_{i-1}$.  The degree of $g$ is the product of the degrees of $g_1,\ldots,g_n$.

The polydegree of $g\notin A$ is the sequence $(d_1,\ldots,d_m)$ of integers $d_j\geq 2$ obtained from the sequence $(\deg g_1,\ldots,\deg g_n)$ by crossing out all the 1's.  Then $\deg g=d_1\cdots d_m$.  The set $G[d_1,\ldots,d_m]$ of all elements of $\cG$ of polydegree $(d_1,\ldots,d_m)$ naturally carries the structure of a complex manifold of dimension $d_1+\cdots+d_m+6$ \cite[Lemma 2.4]{FM1989}, and $\cG$ is the disjoint union
\[ \cG = A \sqcup G[2] \sqcup G[3] \sqcup G[2, 2] \sqcup G[4] \sqcup \cdots. \]
We observe that $\cG$ is thus endowed with what we might call a stratified Oka structure.  (Stratified Oka manifolds were introduced and studied in \cite{FL2014}.)

\begin{proposition}\label{pr:stratified}
Each manifold $G[d_1,\ldots,d_m]$ is Oka.
\end{proposition}

\begin{proof}
By \cite[proof of Lemma 2.4, Lemma 2.10]{FM1989}, $G[d_1,\ldots,d_m]$ is a holomorphic fibre bundle over $\P_1\times \P_1$, whose fibre is biholomorphic to a product of copies of $\C$ and $\C^*$.  Since the base and the fibre are Oka, so is $G[d_1,\ldots,d_m]$.
\end{proof}

To complete the proof of Theorem \ref{th:interpol}, let $D$ be a discrete subset of a Stein space $X$, and $f:D\to\cG$ be a map such that $\deg f:D\to\N$ is bounded.  Then there is a finite partition $D=D_1\sqcup\cdots\sqcup D_n$ such that $f$ maps each $D_j$ into one of the sets $G[d_1,\ldots,d_m]$ or $A$.  Since these sets are connected and Oka, for each $j=1,\ldots,n$, there is a holomorphic map $F_j$ from $X$ into the corresponding set such that $F_j=f$ on $D_j$.  Let $h_j$ be a holomorphic function on $X$ such that $h_j=1$ on $D_j$ and $h_j=0$ on $D\setminus D_j$.

For $\phi\in\cG$, we define an entire curve $\C\to\cG$, $t\mapsto\phi_t$, by the formula
\[ \phi_t(z)=t\phi(0) + t^{-1}(\phi(tz)-\phi(0)), \]
so $\phi_1=\phi$ and $\phi_0=D_0\phi$.  Note that the map $X\to\cG$, $x\mapsto F_j(x)_{h_j(x)}$, is holomorphic.  For $j=1,\ldots,n$, let $\beta_j:D\to GL_2(\C)$ have $\beta_j(x)=\Id$ if $x\in D_j$, and $\beta_j(x)=(D_0 F_i(x))^{-1}$ if $x\in D_i$, $i\neq j$.  Since $GL_2(\C)$ is connected and Oka, $\beta_j$ extends to a holomorphic map $X\to GL_2(\C)$.  Now a holomorphic map $X\to\cG$ extending $f$ is given by the formula
\[ x \mapsto \beta_1(x)\circ F_1(x)_{h_1(x)}\circ\cdots\circ \beta_n(x)\circ F_n(x)_{h_n(x)}.\]


\end{document}